\documentclass{compositio} 
\usepackage[all]{xy}
\usepackage{amsmath}
\usepackage{amssymb}
\usepackage{eucal}
\usepackage{hyperref}

\SelectTips{eu}{}

\newtheorem{lem}{Lemma}[section]
\newtheorem{prop}[lem]{Proposition}
\newtheorem{cor}[lem]{Corollary}
\newtheorem{thm}[lem]{Theorem}

\theoremstyle{remark}
\newtheorem{rem}[lem]{Remark}

\theoremstyle{definition}
\newtheorem{exm}[lem]{Example}

\numberwithin{equation}{section}

\newcommand{\smatrix}[1]{\left[\begin{smallmatrix}#1\end{smallmatrix}\right]}

\renewcommand{\mod}{\operatorname{\mathsf{mod}}\nolimits}
\DeclareMathOperator*{\colim}{colim}

\newcommand{\add}{\operatorname{\mathsf{add}}\nolimits}

\newcommand{\Id}{\operatorname{Id}\nolimits}
\newcommand{\Mod}{\operatorname{\mathsf{Mod}}\nolimits}
\newcommand{\rep}{\operatorname{\mathsf{rep}}\nolimits}
\newcommand{\Rep}{\operatorname{\mathsf{Rep}}\nolimits}

\newcommand{\End}{\operatorname{End}\nolimits}
\newcommand{\Fun}{\operatorname{Fun}\nolimits}
\newcommand{\Hom}{\operatorname{Hom}\nolimits}
\newcommand{\Pol}{\operatorname{Pol}\nolimits}
\newcommand{\RHom}{\operatorname{\mathbf{R}Hom}\nolimits}

\newcommand{\HOM}{\operatorname{\mathcal H \;\!\!\mathit o \mathit m}\nolimits}
\newcommand{\RHOM}{\operatorname{\mathbf{R}\mathcal H \;\!\!\mathit o \mathit m}\nolimits}

\newcommand{\Ext}{\operatorname{Ext}\nolimits}

\newcommand{\sgn}{\operatorname{sgn}\nolimits}

\newcommand{\op}{\mathrm{op}}

\newcommand{\can}{\mathrm{can}}

\newcommand{\comp}{\mathop{\circ}}
\newcommand{\lto}{\longrightarrow}
\newcommand{\xto}{\xrightarrow}
\newcommand{\lotimes}{\otimes^{\mathbf L}}

\def\a{\alpha}
\def\b{\beta}
\def\e{\varepsilon}

\def\g{\gamma}
\def\p{\phi}

\def\s{\sigma}

\def\om{\omega}

\def\la{\lambda}
\def\De{\Delta}
\def\Ga{\Gamma}
\def\La{\Lambda}

\def\C{{\mathsf C}}
\def\D{{\mathsf D}}

\def\K{{\mathsf K}}

\def\M{{\mathsf M}}

\def\P{{\mathsf P}}

\def\bbF{{\mathbb F}}

\def\bbZ{{\mathbb Z}}

\def\frm{{\mathfrak m}}
\def\frp{{\mathfrak p}}

\def\frS{{\mathfrak S}}

\def\sa{\smatrix{\a}}
\def\so{\smatrix{\om}}
\def\sao{\smatrix{\a\\ \om}}
\def\soa{\smatrix{\om\\ \a}}
\def\soao{\smatrix{\om\\ \a\\ \om}}

\begin{document}
\title{Koszul, Ringel, and Serre duality\\ for strict polynomial
  functors}
\dedication{Dedicated to Ragnar-Olaf Buchweitz on the occasion of his
  60th birthday.}

\author{Henning Krause}
\email{hkrause@math.uni-bielefeld.de}
\address{Fakult\"at f\"ur Mathematik\\
Universit\"at Bielefeld\\ D-33501 Bielefeld\\ Germany.}

\begin{abstract}
  This is a report on recent work of Cha{\l}upnik and Touz\'e.  We
  explain the Koszul duality for the category of strict polynomial
  functors and make explicit the underlying monoidal structure which
  seems to be of independent interest. Then we connect this to Ringel
  duality for Schur algebras and describe Serre duality  for strict
  polynomial functors.
\end{abstract}

\maketitle
\tableofcontents 

\section{Introduction}

A Koszul duality for strict polynomial functors has been introduced in recent
work of Cha{\l}upnik \cite{Ch2008} and Touz\'e \cite{To2011}. In this
note we give a detailed report. Our aim is:
\begin{itemize}
\item to make explicit the underlying monoidal structure for the
  category of strict polynomial functors,
\item to explain the connection with Ringel duality for Schur
  algebras,
\item to describe Serre duality for strict polynomial functors in
  terms of Koszul duality, and
\item to remove assumptions on the ring of coefficients.
\end{itemize}

The category of strict polynomial functors was introduced by
Friedlander and Suslin \cite{FS1997} in their work on the cohomology
of finite group schemes. We use an equivalent description in terms of
representations of divided powers, following expositions of Kuhn
\cite{Ku1998}, Pirashvili \cite{Pi2003}, and Bousfield \cite{Bo1967}.

The construction of the Koszul duality can be summarised as follows.
\begin{thm}\label{th:intro}
\pushQED{\qed} 
  Let $k$ be a commutative ring and $d$ a non-negative integer.
  Denote by $\Rep\Ga^d_k$ the category of $k$-linear representations
  of the category of divided powers of degree $d$ over $k$.
\begin{enumerate}
\item The category $\Rep\Ga^d_k$ has a tensor product
\[-\otimes_{\Ga^d_k}-\colon\Rep\Ga^d_k\times\Rep\Ga^d_k\lto\Rep\Ga^d_k\]
which is induced by the tensor product for the category of divided
powers.
\item The left derived tensor functor given by the exterior power
  $\La^d$ of degree $d$ induces for the derived category of
  $\Rep\Ga^d_k$ an equivalence
  \[\La^d\lotimes_{\Ga^d_k}-\colon\D(\Rep\Ga^d_k)
  \stackrel{\sim}\lto\D(\Rep\Ga^d_k).\qedhere\]
\end{enumerate}
\end{thm}

The crucial observation for proving the second part of this theorem is
that classical Koszul duality provides resolutions of the exterior
power $\La^d$ and the symmetric power $S^d$ which give an isomorphism
\begin{equation}\label{eq:square}
\La^d\lotimes_{\Ga^d_k}\La^d\cong S^d.
\end{equation}

Strict polynomial functors are closely related to modules over Schur
algebras. In fact, $\Rep\Ga^d_k$ is equivalent to the category of
modules over the Schur algebra $S_k(n,d)$ for $n\ge d$, by a result of
Friedlander and Suslin \cite{FS1997}. Schur algebras were introduced
by Green \cite{Gr1980} and provide a link between representations of
the general linear groups and the symmetric
groups. Theorem~\ref{th:intro} establishes via transport of structure
a tensor product for the category of modules over the Schur algebra
$S_k(n,d)$; it seems that this tensor product has not been noticed
before.

A Koszul type duality for modules over Schur algebras seems to appear
first in work of Akin and Buchsbaum \cite{AB1988}. For instance, they
construct resolutions of Schur and Weyl modules, and it is Koszul
duality which maps one to the other; see Proposition~\ref{pr:schur}. A
few years later, Ringel \cite{Ri1991} introduced characteristic
tilting modules for quasi-hereditary algebras. These tilting modules
give rise to derived equivalences, relating a quasi-hereditary algebra
and its Ringel dual.  Donkin \cite{Do1993} described those tilting
modules for Schur algebras, and it turns out that Koszul duality as
introduced by Cha{\l}upnik \cite{Ch2008} and Touz\'e \cite{To2011} is
equivalent to Ringel duality for modules over Schur algebras; this is
the content of Theorem~\ref{th:Ringel}.

The Koszul duality $\La^d\lotimes_{\Ga^d_k}-$ is an endofunctor and it
would be interesting to have a description of its powers. It is
somewhat surprising that its square is a Serre duality functor in the
sense of \cite{RV2002}; this is another consequence of the isomorphism
\eqref{eq:square} and the content of Corollary~\ref{co:serre}. 

\subsection*{Acknowledgements}

My interest in the subject of these notes originates from lectures of
Greg Stevenson (at a summer school in Bad Driburg in August 2011) and
of Antoine Touz\'e (at a conference in Strasbourg in November 2011). I
am grateful to both of them for making me aware of the beautiful work
of Cha{\l}upnik and Touz\'e. In addition, I would like to thank Dave
Benson, Ivo Dell'Ambrogio, Greg Stevenson, and Dieter Vossieck for
numerous helpful comments on this subject.

\subsection*{Convention}
Throughout this work we fix a commutative ring $k$.

\section{Divided powers and strict polynomial functors}

The category of strict polynomial functors was introduced by
Friedlander and Suslin \cite{FS1997}.  We use an equivalent
description in terms of representations of divided powers, following
expositions of Kuhn \cite{Ku1998}, Pirashvili \cite{Pi2003}, and
Bousfield \cite{Bo1967}.  Note that divided powers have their origin in the
computation of the homology of Eilenberg--MacLane spaces \cite{Ca1954,
  EM1954}. Some classical references on polynomial functors and polynomial
representations are \cite{ABW1982,Gr1980,MD1995}.

\subsection*{Finitely generated projective modules}
Let $\P_k$ denote the category of finitely generated projective
$k$-modules. Given $V,W$ in $\P_k$, we write $V\otimes W$ for their
tensor product over $k$ and $\Hom(V,W)$ for the group of $k$-linear
maps $V\to W$. This provides two bifunctors
\begin{align*}
-\otimes -&\colon\P_k\times\P_k\lto\P_k\\
\Hom(-,-)&\colon(\P_k)^\op\times\P_k\lto\P_k
\end{align*}
with a natural isomorphism
\[\Hom_{\P_k}(U\otimes V,W)\cong\Hom_{\P_k}(U,\Hom(V,W)).\]

The functor sending $V$ to $V^*=\Hom(V,k)$ yields a duality 
\[(\P_k)^\op \stackrel{\sim}\lto\P_k.\] 

Note that for $U,V,W,V',W'$ in $\P_k$  there are natural isomorphisms
\begin{align}
\label{eq:piso1} V^*\otimes W&\cong \Hom(V,W)\\
\label{eq:piso2} \Hom(U,V)\otimes W&\cong\Hom(U,V\otimes W)\\
\label{eq:piso3} \Hom(V,W)\otimes\Hom(V',W')&\cong\Hom(V\otimes V',W\otimes W').
\end{align}

\subsection*{Divided and symmetric powers}

Fix a positive integer $d$ and denote by $\frS_d$ the symmetric group
permuting $d$ elements. For each $V\in\P_k$, the group $\frS_d$ acts
on $V^{\otimes d}$ by permuting the factors of the tensor product.
Denote by $\Ga^dV$ the submodule $(V^{\otimes d})^{\frS_d}$ of
$V^{\otimes d}$ consisting of the elements which are invariant under
the action of $\frS_d$; it is called the module of \emph{divided
  powers of degree $d$}.\footnote{The original definition of the
  module $\Ga^dV$ of divided powers is different; it is, however,
  isomorphic to the module of symmetric tensors which is used here;
  see \cite[IV.5, Exercise~8]{Bo1981}.}  The maximal quotient of
$V^{\otimes d}$ on which $\frS_d$ acts trivially is denoted by $S^dV$
and is called the module of \emph{symmetric powers of degree $d$}.
Set $\Ga^0V=k$ and $S^0V=k$.

From the definition, it follows that $(\Ga^dV)^*\cong S^d (V^*)$.  Note
that $S^d V$ is a free $k$-module provided that $V$ is free; see
\cite[III.6.6]{Bo1970}. Thus $\Ga^d V$ and $S^d V$ belong to $\P_k$ for
all $V\in\P_k$, and we obtain functors $\Ga^d,S^d\colon\P_k\to\P_k$.

For further material on divided and symmetric powers, see \cite{ABW1982, Bo1981, Ro1963}.

\subsection*{The category of divided powers}
We consider the category $\Ga^d\P_k$ which is defined as follows. The
objects are the finitely generated projective $k$-modules and for two
objects $V,W$ set
 \[\Hom_{\Ga^d\P_k}(V,W)=\Ga^d\Hom(V,W).\]
 Note that this identifies with $\Hom(V^{\otimes d},W^{\otimes
   d})^{\frS_d}$ via the isomorphism \eqref{eq:piso3}, where $\frS_d$
 acts on $\Hom(V^{\otimes d},W^{\otimes d})$
 via \[(f\s)(v_1\otimes\cdots\otimes
 v_d)=f(v_{\s(1)}\otimes\cdots\otimes v_{\s(d)}).\] Using this
 identification one defines the composition of morphisms in
 $\Ga^d\P_k$.

\begin{exm}
  Let $n$ be a positive integer and set $V=k^n$. Then
  $\End_{\Ga^d\P_k}(V)$ is isomorphic to the \emph{Schur algebra}
  $S_k(n,d)$ as defined by Green \cite[Theorem~2.6c]{Gr1980}. We view
  this isomorphism as an identification.
\end{exm}

The bifunctors $-\otimes-$ and $\Hom(-,-)$ for $\P_k$ induce
bifunctors 
\begin{align*}
-\otimes -&\colon\Ga^d\P_k\times\Ga^d\P_k\lto\Ga^d\P_k\\
\Hom(-,-)&\colon(\Ga^d\P_k)^\op\times\Ga^d\P_k\lto\Ga^d\P_k
\end{align*}
with a natural isomorphism
\[\Hom_{\Ga^d\P_k}(U\otimes V,W)\cong\Hom_{\Ga^d\P_k}(U,\Hom(V,W)).\]
More precisely, the tensor product for $\Ga^d\P_k$ coincides on
objects with the one for $\P_k$, while on morphisms it is for pairs
$V,V'$ and $W,W'$ of objects the composite
\begin{multline*}
\Ga^d\Hom(V,V')\times\Ga^d\Hom(W,W')\to
\Ga^d(\Hom(V,V')\otimes\Hom(W,W'))\\ \xto{\sim}\Ga^d\Hom(V\otimes
W,V'\otimes W')
\end{multline*}
where the second map is induced by \eqref{eq:piso3}.

The duality for $\P_k$ induces a duality \[(\Ga^d\P_k)^\op\stackrel{\sim}\lto \Ga^d\P_k.\]

\subsection*{Strict polynomial functors}

Let $\M_k$ denote the category of $k$-modules. We study the category
of $k$-linear \emph{representations} of $\Ga^d\P_k$. This is by
definition the category of $k$-linear functors $\Ga^d\P_k\to\M_k$ and
we write by slight abuse of notation
\[\Rep\Ga^d_k=\Fun_k(\Ga^d\P_k,\M_k).\]
The set of morphisms between two representations $X,Y$ in $\Rep\Ga^d_k$
is denoted by $\Hom_{\Ga^d_k}(X,Y)$.

The representations of $\Ga^d\P_k$ form an abelian category, where
(co)kernels and (co)products are computed pointwise over $k$.

A representation $X$ is called \emph{finite} when $X(V)$
is finitely generated projective for each $V\in\Ga^d\P_k$. 
It is sometimes convenient to restrict to the full subcategory
\[\rep\Ga^d_k=\Fun_k(\Ga^d\P_k,\P_k)\] of finite
representations; it is an extension closed subcategory of
$\Rep\Ga^d_k$ and therefore a Quillen exact category \cite{Q}.

\begin{rem}
  A representation $X\in\Rep\Ga^d_k$ is by definition a pair of
  functions, the first of which assigns to each $V\in\P_k$ a $k$-module
  $X(V)$ and the second  assigns to each pair $V,W\in\P_k$ a
  $k$-linear map
  \[X_{V,W}\colon\Ga^d\Hom(V,W)\lto\Hom(X(V),X(W)).\] The map
  $X_{V,W}$ admits an interpretation which is based on classical
  properties of symmetric tensors and divided powers, to be found in
  \cite{Bo1981, Ro1963}.
  Given a pair of $k$-modules $M,N$ with $M\in\P_k$, there is a
  canonical $k$-linear map
\[\Hom(\Ga^d M, N)\lto\Pol^d(M,N)\]
where $\Pol^d(M,N)$ denotes the $k$-module consisting of
\emph{homogeneous polynomial maps of degree $d$} from $M$ to $N$; it
is an isomorphism when $k$ is an infinite integral domain and $N$ is
torsion-free \cite[IV.5.9]{Bo1981}. On the other hand,
there is a canonical bijection
\[\Hom(\Ga^d M, N)\lto P^d(M,N)\] where $P^d(M,N)$ denotes the set of
\emph{homogeneous polynomial laws of degree $d$} of $M$ in $N$, by
\cite[IV.5, Exercise~10]{Bo1981}.  These observations explain the term
\emph{strict polynomial functor} used by Friedlander and Suslin in
\cite[\S2]{FS1997}.
\end{rem}

\subsection*{The Yoneda embedding}
The Yoneda embedding
\[(\Ga^d\P_k)^\op\lto\Rep\Ga^d_k,\quad V\mapsto\Hom_{\Ga^d\P_k}(V,-)\]
identifies $\Ga^d\P_k$ with the full subcategory
consisting of the representable functors. For  $V\in\Ga^d\P_k$ we write
\[\Ga^{d,V}=\Hom_{\Ga^d\P_k}(V,-).\]
For $X\in\Rep\Ga^d_k$ there is the Yoneda isomorphism
\[\Hom_{\Ga^d_k}(\Ga^{d,V},X)\xto{\sim} X(V)\] 
and it follows that $\Ga^{d,V}$ is a projective
object in $\Rep\Ga^d_k$. For  $W\in\Ga^d\P_k$ this yields
\[\Hom_{\Ga^d_k}(\Ga^{d,V},\Ga^{d,W})\cong\Hom_{\Ga^d\P_k}(W,V)=\Ga^d\Hom(W,V).\]
A representation $X$ is \emph{finitely
  generated} if there are objects $V_1,\ldots,V_r \in\Ga^d \P_k$ and an epimorphism
\[\Ga^{d,V_1}\oplus\cdots\oplus\Ga^{d,V_r}\lto X\]

Note that each $X$ in $\Rep\Ga^d_k$ can be written canonically as a
colimit of representable functors
\[\colim_{\Ga^{d,V}\to X} \Ga^{d,V} \xto{\sim}X\]
where the colimit is taken over the category of morphisms
$\Ga^{d,V}\to X$ and $V$ runs through the objects of $\Ga^d\P_k$.

\subsection*{Duality}
Given a  representation $X\in\Rep\Ga^d_k$, its \emph{dual}  $X^\circ$ is defined by
\[X^\circ(V)=X(V^*)^*.\] This is also known as \emph{Kuhn dual}
\cite{Ku1994}; see \cite[Section~2.7]{Gr1980} for its use in representation
theory.

For each pair of $k$-modules $V,W$ there is a natural isomorphism
\[\Hom(V,W^*)\cong\Hom(W,V^*).\]
This induces for all $X,Y\in\Rep\Ga^d_k$ a natural isomorphism
\begin{equation}\label{eq:dual}
\Hom_{\Ga^d_k}(X,Y^\circ)\cong\Hom_{\Ga^d_k}(Y,X^\circ)
\end{equation}
and the evaluation morphism $X\to X^{\circ\circ}$, which is an
isomorphism when $X$ is finite.
\begin{exm} The divided power functor $\Ga^d$ and the symmetric power
  functor $S^d$ belong to $\Rep\Ga^d_k$. In fact
\[\Ga^d=\Hom_{\Ga^d\P_k}(k,-)\quad\text{and}\quad S^d=(\Ga^d)^\circ.\]
\end{exm}

\subsection*{The internal tensor product}

The category of representations of $\Ga^d\P_k$ inherits a tensor
product from the tensor product for $\Ga^d\P_k$.  We provide a
construction in terms of Kan extensions which is also known as
Day convolution \cite{Da1970, IK1986}. The internal Hom
functor appears in work of Touz\'e  \cite[\S2]{To2010}.

\begin{prop}\label{pr:tensor}
The bifunctors $-\otimes-$ and $\Hom(-,-)$  for $\Ga^d\P_k$ induce
via the Yoneda embedding bifunctors 
\begin{align*}
-\otimes_{\Ga^d_k}-&\colon\Rep\Ga^d_k\times\Rep\Ga^d_k\lto\Rep\Ga^d_k\\
\HOM_{\Ga^d_k}(-,-) &\colon
(\Rep\Ga^d_k)^\op\times\Rep\Ga^d_k\lto\Rep\Ga^d_k
\end{align*}
 with  a natural isomorphism
\[\Hom_{\Ga^d_k}(X\otimes_{\Ga^d_k}Y,Z)\cong\Hom_{\Ga^d_k}(X,\HOM_{\Ga^d_k}(Y,Z)).\]
\end{prop}

To be precise, one requires for $V,W\in\Ga^d\P_k$ that 
\begin{align}
\label{eq:intern1}\Ga^{d,V}\otimes_{\Ga^d_k}\Ga^{d,W}&=\Ga^{d,V\otimes W}\\
\label{eq:intern2}\HOM_{\Ga^d_k}(\Ga^{d,V},\Ga^{d,W})&=\Ga^{d,\Hom(V,W)}
\end{align} 
and this determines both bifunctors.

\begin{proof}
Given $X,Y\in\Rep\Ga^d_k$ and $V\in\Ga^d\P_k$, one defines
\begin{align*}
\Ga^{d,V}\otimes_{\Ga^d_k} Y&=\colim_{\Ga^{d,W}\to Y} \Ga^{d,V\otimes
  W}\\
\HOM_{\Ga^d_k}(\Ga^{d,V}, Y)&=\colim_{\Ga^{d,W}\to Y} \Ga^{d,\Hom(V,W)}
\intertext{and then}
X\otimes_{\Ga^d_k} Y&=\colim_{\Ga^{d,V}\to X}
\Ga^{d,V}\otimes_{\Ga^d_k} Y\\
\HOM_{\Ga^d_k}(X, Y)&=\lim_{\Ga^{d,V}\to
  X}\HOM_{\Ga^d_k}(\Ga^{d,V},Y).
\end{align*}

For the adjunction isomorphism, write $X=\colim_\a X_\a$,
$Y=\colim_\b Y_\b$, and $Z=\colim_\g Z_\g$ as colimits of
representable functors as before. Then we have
\begin{align*}
  \Hom_{\Ga^d_k}(X\otimes_{\Ga^d_k}Y,Z)&\cong
  \Hom_{\Ga^d_k}(\colim_\a\colim_\b X_\a\otimes_{\Ga^d_k}Y_\b,\colim_\g Z_\g)\\
  &\cong
  \lim_\a\lim_\b\colim_\g \Hom_{\Ga^d_k}(X_\a\otimes_{\Ga^d_k}Y_\b,Z_\g)\\
  &\cong
  \lim_\a\lim_\b\colim_\g \Hom_{\Ga^d_k}(X_\a,\HOM_{\Ga^d_k}(Y_\b,Z_\g))\\
  &\cong
  \Hom_{\Ga^d_k}(\colim_\a X_\a,\lim_\b\colim_\g\HOM_{\Ga^d_k}(Y_\b,Z_\g))\\
  &\cong \Hom_{\Ga^d_k}(X,\HOM_{\Ga^d_k}(Y,Z)).\tag*{\qed}
\end{align*}
\renewcommand{\qedsymbol}{}
\end{proof}

The tensor product and the internal Hom functor enjoy the usual
categorical properties. For instance, the tensor product is right
exact and can be computed using projective presentations.

\begin{lem}\label{le:representable}
Given $X\in\Rep\Ga^d_k$ and $V\in\Ga^d\P_k$, there are natural isomorphisms
\begin{align*}
\Ga^{d,V}\otimes_{\Ga^d_k}X&\cong X\comp\Hom(V,-)\\ 
\HOM_{\Ga^d_k}(\Ga^{d,V},X)&\cong X\comp V\otimes -.
\end{align*}
\end{lem}
\begin{proof}
  The isomorphisms are clear from \eqref{eq:intern1} and
  \eqref{eq:intern2} when $X$ is representable. The general case
  follows since $\Ga^{d,V}\otimes_{\Ga^d_k}-$ and
  $\HOM_{\Ga^d_k}(\Ga^{d,V},-)$ preserve colimits.
\end{proof}

\subsection*{Hom-tensor identities}
We collect some basic identities for the internal products
$-\otimes_{\Ga^d_k}-$ and $\HOM_{\Ga^d_k}(-,-)$.
\begin{lem}\label{le:tensor-hom}
Given $X,Y,Z\in\Rep\Ga^d_k$, the natural morphism
\[\HOM _{\Ga^d_k} (X,Y) \otimes_{\Ga^d_k}
Z\lto\HOM_{\Ga^d_k}(X,Y\otimes_{\Ga^d_k} Z)\] is an isomorphism
provided that $X$ or $Z$ is finitely generated
projective.\footnote{This means that finitely generated projective
  objects are \emph{strongly dualisable} in the sense of
  \cite[III.1]{LMSM}. We refer to these notes for a discussion of
  further tensor categorical properties.}
\end{lem}
\begin{proof}
  The isomorphism is clear from \eqref{eq:piso2} when all objects are
  representable functors. The general case follows by writing functors
  as colimits of representable functors, keeping in mind that
  $\Ga^{d,V}\otimes_{\Ga^d_k}-$ and $\HOM_{\Ga^d_k}(\Ga^{d,V},-)$
  preserve colimits.
\end{proof}

\begin{lem}\label{le:dual}
  Let $X,Y$ be in $\Rep\Ga^d_k$ and suppose that $X$ is finitely
  presented. Then there are natural isomorphisms
\begin{align*} 
X\otimes_{\Ga^d_k}Y^\circ&\cong \HOM_{\Ga^d_k}(X,Y)^\circ\\
(X\otimes_{\Ga^d_k}Y)^\circ&\cong \HOM_{\Ga^d_k}(X,Y^\circ).
\end{align*}
\end{lem}
\begin{proof}
  Suppose first that $X$ is representable. Then the isomorphism
  follows from Lemma~\ref{le:representable}, using
  \eqref{eq:piso1}. For a representation $X$ that admits a
  presentation \[\Ga^{d,V_1}\lto\Ga^{d,V_0}\lto X\lto 0\] one uses exactness.
\end{proof}

\begin{lem}\label{le:dual1}
Given $X,Y\in\Rep\Ga^d_k$, there is a natural isomorphism 
\[\HOM_{\Ga^d_k}(X,Y^\circ)\cong\HOM_{\Ga^d_k}(Y,X^\circ).\]
\end{lem}
\begin{proof}
  Choose a representable functor $P$. Using \eqref{eq:dual} and
  Lemma~\ref{le:dual}, we have
\begin{align*}
  \Hom_{\Ga^d_k}(P,\HOM_{\Ga^d_k}(X,Y^\circ))&\cong
  \Hom_{\Ga^d_k}(P\otimes_{\Ga^d_k}X,Y^\circ)\\
  &\cong  \Hom_{\Ga^d_k}(Y,(P\otimes_{\Ga^d_k}X)^\circ)\\
  &\cong  \Hom_{\Ga^d_k}(Y,\HOM_{\Ga^d_k}(P,X^\circ))\\
 &\cong  \Hom_{\Ga^d_k}(P,\HOM_{\Ga^d_k}(Y,X^\circ)).
\end{align*}
The assertion now follows from Yoneda's lemma.
\end{proof}

\subsection*{The external tensor product}
Let $d,e$ be non-negative integers and denote by
$(\Ga^d\otimes\Ga^e)\P_k$ the category having as objects the
finitely generated projective $k$-modules and as morphisms $V\to W$
\[\Ga^d\Hom(V,W)\otimes \Ga^e\Hom(V,W).\]
The inclusion $\frS_{d}\times\frS_e\subseteq \frS_{d+e}$ induces for each $V\in\P_k$ an
inclusion \[\Ga^{d+e}V=(V^{\otimes d+e})^{\frS_{d+e}}\subseteq (V^{\otimes d+e})^{\frS_d\times\frS_e}
\cong (V^{\otimes d})^{\frS_d}\otimes (V^{\otimes e})^{\frS_e}=\Ga^dV\otimes\Ga^eV.
\]
This yields a $k$-linear
functor \[C^{d,e}\colon\Ga^{d+e}\P_k\lto(\Ga^d\otimes\Ga^e)\P_k\]
which is the identity on objects.  It induces a tensor
product \[-\otimes -\colon\Rep\Ga^d_k\times \Rep\Ga^e_k\lto
\Rep\Ga^{d+e}_k.\] To be precise, one defines for $X\in\Rep\Ga^d_k$,
$Y\in\Rep\Ga^e_k$, and $V\in\P_k$
\[(X\otimes Y)(V)=X(V)\otimes Y(V)\] and
$X\otimes Y$ acts on morphisms via $C^{d,e}$.

Note that the external tensor product is exact when restricted to
finite representations and that \[(X\otimes Y)^\circ\cong
X^\circ\otimes Y^\circ\] for finite representations $X,Y$.

For integers $d\ge 0$ and $n >0$, we denote by $\La(n,d)$ the set of
sequences $\la=(\la_1,\la_2,\ldots,\la_n)$ of non-negative integers
such that $\sum\la_i=d$. Given $\la\in\La(n,d)$, we
write \[\Ga^\la=\Ga^{\la_1}\otimes\cdots\otimes\Ga^{\la_n}\quad
\text{and}\quad S^\la=S^{\la_1}\otimes\cdots\otimes S^{\la_n}.\] Note
that
\begin{equation}\label{eq:tensor}
\Ga^{(1,\ldots,1)}\cong\otimes ^n\cong S^{(1,\ldots,1)}.
\end{equation}

\subsection*{Graded representations} 

It is sometimes convenient to consider the category 
\[\prod_{d\ge 0}\Rep\Ga^d_k\]
consisting of graded representations $X=(X^0,X^1,X^2,\ldots)$. An
example is for each $V\in\P_k$ the representation
\[\Ga^V=(\Ga^{0,V},\Ga^{1,V},\Ga^{2,V},\ldots).\] There are two tensor
products for graded representations. Given representations $X,Y$, they
are defined in degree $d$ by
\[(X\otimes Y)^d=\bigoplus_{i+j=d}X^i\otimes Y^j
\quad\text{and}\quad
(X\otimes_{\Ga_k} Y)^d=X^d\otimes_{\Ga^d_k} Y^d.\]

\subsection*{Decomposing divided powers}

Given $V\in\P_k$, let $SV=\bigoplus_{d\ge 0}S^dV$ denote the
\emph{symmetric algebra}. This gives a functor from $\P_k$ to the
category of commutative $k$-algebras which preserves coproducts. Thus
\[S V\otimes S W\cong S(V\oplus W)\]
and therefore by duality
\[\Ga V\otimes \Ga W\cong\Ga (V\oplus W).\]
This yields an isomorphism of graded representations
\[\Ga^{V}\otimes\Ga^{W}\cong \Ga^{V\oplus W}\]
since for each $d\ge 0$
\[(\Ga^{V}\otimes\Ga^{W})^d\cong \bigoplus_{i+j=d}(\Ga^{i,V}\otimes \Ga^{j,W})
\cong\Ga^{d,V\oplus W}.\]
Thus for each positive integer $n$, one obtains a decomposition
\begin{equation*}
\Ga^{d,k^n}=\bigoplus_{i=0}^d(\Ga^{d-i,k^{n-1}}\otimes \Ga^i)
\end{equation*}
and using induction a canonical decomposition
\begin{equation}\label{eq:decomp}
\Ga^{d,k^n}=\bigoplus_{\la\in\La(n,d)}\Ga^{\la}.
\end{equation}

This decomposition of divided powers has the following immediate
consequence.

\begin{prop}\label{pr:proj}
  The category of finitely generated projective objects of
  $\Rep\Ga^d_k$ is equivalent to the category of direct summands of
  finite direct sums of representations of the form $\Ga^\la$, where
  $\la=(\la_1,\ldots,\la_n)$ is a sequence of non-negative integers
  satisfying $\sum\la_i=d$ and $n$ is a positive integer. 
\end{prop}
\begin{proof}
  This follows from Yoneda's lemma and the fact that each
  representable functor admits a decomposition \eqref{eq:decomp} into
  representations of the form $\Ga^\la$.
\end{proof}

\subsection*{Representations of Schur algebras}

Strict polynomial functors and modules over Schur algebras are closely
related by a result due to Friedlander and Suslin
\cite[Theorem~3.2]{FS1997}; it is an immediate consequence of
Proposition~\ref{pr:proj}.  Given any ring $A$, we denote by $\Mod A$
the category of $A$-modules.

\begin{thm}\label{th:schur}
Let $d,n$ be positive integers. Evaluation at $k^n$ induces a
  functor $\Rep\Ga^d_k\to \Mod S_k(n,d)$ which is an
  equivalence if $n\ge d$.
\end{thm}
\begin{proof}
  Let $P\in\Rep\Ga^d_k$ be a \emph{small projective generator}, that is, $P$
  is a projective object and the functor $\Hom_{\Ga^d_k}(P,-)$ is
  faithful and preserves set-indexed direct
  sums. Then \[\Hom_{\Ga^d_k}(P,-)\colon \Rep\Ga^d_k\lto
  \Mod \End_{\Ga^d_k}(P)\] is an equivalence.

  From Yoneda's lemma it follows that the representable functors
  $\Ga^{d,V}$ with $V\in\P_k$ form a family of small projective
  generators of $\Rep\Ga^d_k$. Now assume $n\ge d$. The decomposition
  \eqref{eq:decomp} then implies that the functors $\Ga^\la$ with
  $\la\in\La(n,d)$ form a family of small projective generators. Thus
  $P=\Ga^{d,k^n}$ is a small projective generator since each $\Ga^\la$
  occurs as a direct summand. It remains to observe that
  $\End_{\Ga^d_k}(P)=S_k(n,d)$ and that $\Hom_{\Ga^d_k}(P,-)$ equals
  evaluation at $k^n$, by Yoneda's lemma.
\end{proof}

\begin{rem}
  Let $n\ge d$. The evaluation functor $\Rep\Ga^d_k\xto{\sim} \Mod
  S_k(n,d)$ restricts to an equivalence $\rep\Ga^d_k\xto{\sim} \mod S_k(n,d)$,
  where $\mod S_k(n,d)$ denotes the full subcategory of modules that
  are finitely generated projective over $k$.
\end{rem}

\begin{rem}
  The tensor product for the category $\Rep\Ga^d_k$ from
  Proposition~\ref{pr:tensor} induces via transport of structure a
  tensor product for the category $\Mod S_k(n,d)$ when $n\ge d$. It
  seems that this tensor product has not been noticed before, despite
  the fact that polynomial representations of the general linear
  groups have been studied for more than a hundred years.
\end{rem}

\subsection*{Representations of  symmetric groups}

Schur--Weyl duality yields a relation between representations of the
general linear groups and representations of the symmetric groups. In
our context this takes the following form. Let $\omega=(1,\ldots,1)$
be a sequence of length $d$. Then $\Ga^{\omega}$ is the functor taking
$V$ to $V^{\otimes d}$ and \[\End_{\Ga^d_k}(\Ga^{\omega})\cong
k\frS_d,\] where $k\frS_d$ denotes the group algebra of the symmetric
group $\frS_d$. This observation gives rise to the functor
$\Hom_{\Ga^d_k}(\Ga^{\omega},-)\colon \Rep \Ga^d_k\to \Mod k\frS_d$,
which appears in Schur's thesis \cite[Sections~III, IV]{Sc1901}; see
also \cite[\S6]{Gr1980}. The functor is not relevant for the rest
of this work, but we mention that it has a fully faithful left adjoint
$\Mod k\frS_d\to\Rep\Ga^d_k$ taking $X\in\Mod k\frS_d$ to the functor
\[V\mapsto X\otimes_{k\frS_d}\Hom_{\Ga^d_k}(\Ga^{d,V},\Ga^{\omega})\]
and a fully faithful right adjoint
$\Mod k\frS_d\to\Rep\Ga^d_k$ taking $X\in\Mod k\frS_d$ to the functor
\[V\mapsto \Hom_{k\frS_d}(\Hom_{\Ga^d_k}(\Ga^{\omega},\Ga^{d,V}), X).\]

\section{Exterior powers and Koszul duality}

Koszul duality expresses the intimate homological relation between symmetric and
exterior powers. For strict polynomial functors, Koszul duality has
been introduced by Cha\l upnik \cite[\S2]{Ch2008} and Touz\'e
\cite[\S3]{To2011} as the derived functor of $\HOM_{\Ga^d_k}(\La^d,-)$. In
this section we treat the non-derived version. A crucial ingredient is
the compatibility of the internal and external tensor products; for this we
follow closely Touz\'e \cite[\S2]{To2010}. For some earlier work
relating symmetric and exterior powers for Schur algebras, we refer to
\cite{Ak1989, AB1988, ABW1982}.

\subsection*{Exterior powers}

Given $V\in\P_k$, let $\La V=\bigoplus_{d\ge 0}\La^d V$ denote the
\emph{exterior algebra}, which is obtained from the tensor algebra
$TV=\bigoplus_{d\ge 0}V^{\otimes d}$ by taking the quotient with
respect to the ideal generated by the elements $v\otimes v$, $v\in
V$. 

For each $d\ge 0$, the $k$-module $\La^dV$ is free provided that $V$
is free; see \cite[III.7.8]{Bo1970}. Thus $\La^d V$ belongs to $\P_k$
for all $V\in\P_k$, and this gives a functor $\Ga^d\P_k\to\P_k$, since
the ideal generated by the elements $v\otimes v$ is invariant under
the action of $\mathfrak S_d$ on $V^{\otimes d}$. There is a natural
isomorphism \[\La^d (V^*)\cong (\La^d V)^*\] induced by
$(f_1\wedge\cdots\wedge f_d)(v_1\wedge\cdots\wedge v_d)=\det
(f_i(v_j))$, and therefore $(\La^d)^\circ\cong\La^d$.

For each $V\in\Ga^d\P_k$, we use the notation
\[\La^{d,V}=\La^d\comp\Hom(V,-)\cong\La^d\otimes_{\Ga^d_k}\Ga^{d,V}\]
and this gives a graded representation
\[\La^V=(\La^{0,V},\La^{1,V},\La^{2,V},\ldots).\] 

Note that multiplication in $\La V$ is graded commutative and that
\[\La V\otimes\La W\cong\La(V\oplus W).\] 
This yields an isomorphism of graded representation
\begin{equation}\label{eq:ext}
  \La^V\otimes\La^W\cong\La^{V\oplus W}.
\end{equation}

\subsection*{Internal versus external tensor product}

We investigate the compatibility of internal and external tensor
product.

\begin{lem}
Let $X_1,X_2\in\Rep\Ga^d_k$ and $Y_1,Y_2\in\Rep\Ga^e_k$. Then there is
a natural morphism
\begin{equation*}\label{eq:phi}
\tau\colon (X_1\otimes_{\Ga^d_k}X_2)\otimes (Y_1\otimes_{\Ga^e_k}Y_2)
\lto (X_1\otimes Y_1)\otimes_{\Ga^{d+e}_k} (X_2\otimes Y_2).
\end{equation*}
\end{lem}
\begin{proof}
  Consider first the case that $X_2=\Ga^{d,V}$ and $Y_2=\Ga^{e,W}$ for
  some $V,W\in\P_k$. From the decomposition \eqref{eq:decomp} it
  follows that $\Ga^{d,V}\otimes \Ga^{e,W}$ is canonically a direct
  summand of $\Ga^{d+e,V\oplus W}$.  Also, we use the description of
  $\Ga^{d,V}\otimes_{\Ga^d_k}-$ from Lemma~\ref{le:representable}. The
  projections $V\oplus W\to V$ and $V\oplus W\to W$ induce a
  morphism \[X_1\comp\Hom(V,-)\otimes
  Y_1\comp\Hom(W,-)\lto (X_1\otimes Y_1)\comp\Hom(V\oplus W,-)\] and
  composition with the canonical projection
\[(X_1\otimes Y_1)\comp\Hom(V\oplus W,-)\lto  (X_1\otimes
Y_1)\otimes_{\Ga^{d+e}_k} (\Ga^{d,V}\otimes\Ga^{e,W})\] gives $\tau$
which is natural in $V\in\Ga^d\P_k$ and $W\in\Ga^e\P_k$.
This last observation yields  $\tau$ for general $X_2$ and $Y_2$.
\end{proof}

\begin{lem}\label{le:int-ext}
Let  $X\in\Rep\Ga^d_k$ and $Y\in\Rep\Ga^e_k$. Then the
composite
\[(\La^d\otimes_{\Ga^d_k}X)\otimes (\La^e\otimes_{\Ga^e_k}Y) \xto{\tau}
(\La^d\otimes \La^e)\otimes_{\Ga^{d+e}_k} (X\otimes Y)\to
\La^{d+e}\otimes_{\Ga^{d+e}_k} (X\otimes Y)\] (where the second map is
induced by the multiplication $\La^d\otimes\La^e\to\La^{d+e}$) is a
natural isomorphism.
\end{lem}
\begin{proof}
  Consider first the case that $X=\Ga^{d,V}$ and $Y=\Ga^{e,W}$ for
  some $V,W\in\P_k$.  From \eqref{eq:ext} we have an isomorphism of
  graded representations
\begin{align*}
(\La\otimes_{\Ga_k}\Ga^{V})\otimes (\La\otimes_{\Ga_k}\Ga^{W})&\cong
\La^{V}\otimes\La^{W}\\
&\cong \La^{V\oplus W}\\
&\cong\La\otimes_{\Ga_k}\Ga^{V\oplus W}\\
&\cong\La\otimes_{\Ga_k}(\Ga^{V}\otimes\Ga^{W})
\end{align*}
which restricts in degree $d+e$ to an
isomorphism
\[\bigoplus_{i+j=d+e}\La^{i,V}\otimes\La^{j,W}\stackrel{\sim}\lto\La^{d+e}\otimes_{\Ga^{d+e}_k}
\left(\bigoplus_{i'+j'=d+e}\Ga^{i',V}\otimes\Ga^{j',W}\right).\] The
$(i,j)=(d,e)$ summand then maps onto the $(i',j')=(d,e)$ summand
(requires checking). This establishes the isomorphism for the case
$X=\Ga^{d,V}$ and $Y=\Ga^{e,W}$.  Naturality in $V,W$ yields the
general case.
\end{proof}

\begin{rem}
  The multiplication maps $\La^d\otimes \La^e\to\La^{d+e}$ and
  $\La^e\otimes \La^d\to\La^{d+e}$ are equal up to $(-1)^{de}\s$,
  where $\s\colon\La^d\otimes\La^e\xto{\sim}\La^e\otimes\La^d$ permutes the
  factors of the tensor product. This sign comes from the graded
  commutativity of the exterior algebra and appears as well in the
  isomorphism of Lemma~\ref{le:int-ext} when the factors $X$ and $Y$
  are permuted.
\end{rem}

For a sequence $\la=(\la_1,\la_2,\ldots,\la_n)$ of non-negative
integers we write
\[\La^\la=\La^{\la_1}\otimes\cdots\otimes\La^{\la_n}.\]

\begin{prop}\label{pr:int-ext}
For each sequence $\la\in\La(n,d)$ there is an isomorphism
\[\La^\la\cong\La^d\otimes_{\Ga^d_k}\Ga^\la.\]
\end{prop}
\begin{proof}
Apply Lemma~\ref{le:int-ext}.
\end{proof}

\subsection*{Koszul duality}

We compute $\La^d\otimes_{\Ga^d_k}-$ and $\HOM_{\Ga^d_k}(\La^d,-)$ using the
relation between exterior and symmetric powers.

\begin{lem}\label{le:pres}
  Given $d\ge 1$ and $V\in\P_k$, there are exact sequences
\begin{gather*}
  \bigoplus_{i=1}^{d-1} V^{\otimes i-1}\otimes\Ga^2
  V\otimes V^{\otimes d-i-1}\xto{1\otimes\Delta\otimes 1}V^{\otimes
    d}\lto\La^d V\lto 0\\
  \bigoplus_{i=1}^{d-1} V^{\otimes i-1}\otimes\La^2
  V\otimes  V^{\otimes d-i-1}\xto{1\otimes\Delta\otimes 1}V^{\otimes
    d}\lto S^d V\lto 0
\end{gather*}
where $\De\colon \Ga^2V\to V\otimes V$ is the component of the
comultiplication which is dual to the multiplication $V^*\otimes
V^*\to S^2(V^*)$, and $\De\colon \La^2V\to V\otimes V$ is the
component of the comultiplication given by $\De(v\wedge w)=v\otimes
w-w\otimes v$.
\end{lem}
\begin{proof}
See \cite[p.~214--216]{ABW1982} or \cite[p.~6]{To1997}.
\end{proof}

\begin{prop}\label{pr:exterior}
Let $d\ge 0$. Then
\[\La^d\otimes_{\Ga^d_k}\La^d\cong S^d.\]
\end{prop}
\begin{proof}
  Tensoring the projective presentation of $\La^d$ in
  Lemma~\ref{le:pres} with $\La^d$ yields an exact sequence which
  identifies with the presentation of $S^d$ via the isomorphism of
  Proposition~\ref{pr:int-ext}. Here, we use the isomorphism \eqref{eq:tensor}.
\end{proof}

\begin{cor}\label{co:exterior}
For each sequence $\la\in\La(n,d)$ there is an isomorphism
\[\La^d\otimes_{\Ga^d_k}\La^\la\cong S^\la.\]
\end{cor}
\begin{proof}
  Combine Lemma~\ref{le:int-ext} and Proposition~\ref{pr:exterior}. To
  be explicit, we have
\[\La^d\otimes_{\Ga^d_k}(\La^{\la_1}\otimes\cdots\otimes\La^{\la_n})\cong 
(\La^{\la_1}\otimes_{\Ga^{\la_1}_k}\La^{\la_1})\otimes\cdots\otimes
(\La^{\la_n}\otimes_{\Ga^{\la_n}_k}\La^{\la_n})\cong
S^\la.\qedhere\]
\end{proof}

For a set $\C$ of objects of an additive category, we denote by
$\add\C$ the full subcategory consisting of all finite direct sums of
objects in $\C$ and their direct summands.

\begin{cor}\label{co:equiv}
  Let $d,n\ge 1$.  The functor $\La^d\otimes_{\Ga^d_k}-$ induces
  equivalences
\begin{align*}
\add\{\Ga^\la\mid\la\in\La(n,d)\}&\stackrel{\sim}\lto
\add\{\La^\la\mid\la\in\La(n,d)\}\\
\add\{\La^\la\mid\la\in\La(n,d)\}&\stackrel{\sim}\lto
\add\{S^\la\mid\la\in\La(n,d)\}
\end{align*}
with quasi-inverses induced by $\HOM_{\Ga^d_k}(\La^d,-)$.
\end{cor}
\begin{proof}
From Lemma~\ref{le:dual} we have for $X\in\Rep\Ga^d_k$ 
\[\HOM_{\Ga^d_k}(\La^d,X)^\circ\cong \La^d\otimes_{\Ga^d_k} X^\circ.\]
The assertion then follows from Proposition~\ref{pr:int-ext} and
Corollary~\ref{co:exterior}, using that $(S^\la)^\circ\cong\Ga^\la$
and $(\La^\la)^\circ\cong\La^\la$ for each sequence $\la$.
\end{proof}

\section{Derived Koszul duality}

In this section we establish the Koszul duality for strict polynomial
functors, working in the unbounded derived category and over an
arbitrary ring of coefficients.  We refer to
\cite[Corollary~2.3]{Ch2008} and \cite[Theorem~3.4]{To2011} for the
corresponding results of Cha\l upnik and Touz\'e.

\subsection*{The derived category of strict polynomial functors}

Let $\D(\Rep\Ga^d_k)$ denote the unbounded derived category of the
abelian category $\Rep\Ga^d_k$. The objects are $\bbZ$-graded
complexes with differential of degree $+1$ and the morphisms are chain
maps with all quasi-isomorphisms inverted.  This is a triangulated
category which admits set-indexed products and coproducts.  The set of
morphisms between two complexes $X,Y$ in $\D(\Rep\Ga^d_k)$ is denoted
by $\Hom_{\D(\Ga^d_k)}(X,Y)$.  An object in $\D(\Rep\Ga^d_k)$ is called
\emph{perfect} if it is isomorphic to a bounded complex of finitely
generated projective objects.

\subsection*{Derived functors}
We construct the derived functors of the internal tensor product and
the internal Hom functor. 

\begin{prop}
  The bifunctors $-\otimes_{\Ga^d_k}-$ and $\HOM_{\Ga^d_k}(-,-)$ for
  $\Rep\Ga^d_k$ have derived functors
\begin{align*}
-\lotimes_{\Ga^d_k}-&\colon\D(\Rep\Ga^d_k)\times
\D(\Rep\Ga^d_k)\lto\D(\Rep\Ga^d_k)\\
\RHOM_{\Ga^d_k}(-,-)&\colon\D(\Rep\Ga^d_k)^\op\times\D(\Rep\Ga^d_k)\lto\D(\Rep\Ga^d_k)
\end{align*}
with a natural isomorphism
\[\Hom_{\D(\Ga^d_k)}(X\lotimes_{\Ga^d_k}
Y,Z)\cong\Hom_{\D(\Ga^d_k)}(X,\RHOM_{\Ga^d_k}(Y,Z)).\] 
\end{prop} 

To compute these functors, one chooses for complexes $X,Y$ in
$\D(\Rep\Ga^d_k)$ a K-projective resolution $\mathbf p Y\to Y$ and a
K-injective resolution $Y\to \mathbf i Y$ (in the sense of \cite{Sp}).
Then one  gets
\begin{align*}
X\lotimes_{\Ga^d_k} Y&= X\otimes_{\Ga^d_k} \mathbf p Y\\
\RHOM_{\Ga^d_k}(X,Y)&=\HOM_{\Ga^d_k}(X,\mathbf i Y).
\end{align*}
This involves  total complexes which are defined in degree $n$ by
\begin{align*}
(X\otimes_{\Ga^d_k} Y)^n&=\bigoplus_{p+q=n} X^p\otimes_{\Ga^d_k} Y^q\\
\HOM_{\Ga^d_k}(X,Y)^n&=\prod_{p+q=n} \HOM_{\Ga^d_k}(X^{-p},Y^q).
\end{align*}
Note that  $\mathbf p Y=Y$ when $Y$ is a bounded  above complex of projective
objects. Also, we have in $\D(\Rep\Ga^d_k)$
\[\HOM_{\Ga^d_k}(\mathbf p X,Y)\xto{\sim}\HOM_{\Ga^d_k}(\mathbf p
X,\mathbf i Y)
\xleftarrow{\sim}\HOM_{\Ga^d_k}(X,\mathbf i Y).
\]
\begin{proof}
  Let $\K(\Rep\Ga^d_k)$ denote the homotopy category of $\Rep\Ga^d_k$
  and consider the quotient functor $Q\colon \K(\Rep\Ga^d_k)\to
  \D(\Rep\Ga^d_k)$ that inverts all quasi-isomorphisms. This functor
  has a left adjoint $\mathbf p$ (taking a complex to its K-projective
  resolution) and a right adjoint $\mathbf i$ (taking a complex to its
  K-injective resolution).

  Given $X$ in $\D(\Rep\Ga^d_k)$, one obtains the following pairs of
  adjoint functors.
\[\xymatrix{\D(\Rep\Ga^d_k)\ar@<0.75ex>[r]^-{\mathbf p}&
  \K(\Rep\Ga^d_k)\ar@<0.75ex>[l]^-{Q}\ar@<0.75ex>[rr]^-{ X\otimes_{\Ga^d_k}-}&&
  \K(\Rep\Ga^d_k)\ar@<0.75ex>[ll]^-{\HOM_{\Ga^d_k}(X,-)}\ar@<0.75ex>[r]^-{Q}&\D(\Rep\Ga^d_k)
  \ar@<0.75ex>[l]^-{\mathbf i} }
\]
The composite from left to right gives the derived functor
$X\lotimes_{\Ga^d_k}-$, while the composite from right to left gives
the derived functor $\RHOM_{\Ga^d_k}(X,-)$.
\end{proof}

\begin{lem}\label{le:per}
Let $P\in\D(\Rep\Ga^d_k)$ be perfect. Then $P\lotimes_{\Ga^d_k} -$ and
$\RHOM_{\Ga^d_k}(P,-)$ preserve set-indexed products and coproducts.
\end{lem}
\begin{proof}
  The objects $X\in\D(\Rep\Ga^d_k)$ such that $X\lotimes_{\Ga^d_k} -$
  and $\RHOM_{\Ga^d_k}(X,-)$ preserve set-indexed products and
  coproducts form a triangulated subcategory of $\D(\Rep\Ga^d_k)$
  which contains all finitely generated projective objects when viewed
  as complexes concentrated in degree zero. It follows that each
  perfect object belongs to this subcategory.
\end{proof}

\subsection*{Duality}
The duality taking $X\in \Rep\Ga^d_k$ to $X^\circ$ has a derived
functor.  For a complex $X=(X^n,d_X^n)$ we
define its \emph{dual} $X^\circ$ by
\[(X^\circ)^n=(X^{-n})^\circ\quad\text{and}\quad d^n_{X^\circ}=(-1)^{n+1} (d_X^{-n-1})^\circ\]
and its
\emph{derived dual} by \[X^\diamond =(\mathbf p X)^\circ.\] 
Note that
\[\HOM_{\Ga^d_k}(X,S^d)\cong X^\diamond.\]

\begin{lem}\label{le:der-dual}
Given $X,Y\in\D(\Rep\Ga^d_k)$, there is a natural isomorphism 
\[\Hom_{\D(\Ga^d_k)}(X,Y^\diamond)\cong\Hom_{\D(\Ga^d_k)}(Y,X^\diamond).\]
\end{lem}
\begin{proof} 
  The assertion means that $(-)^\diamond$ as a functor
  $\D(\Rep\Ga^d_k)^\op\to \D(\Rep\Ga^d_k)$ is self-adjoint.  From
  \eqref{eq:dual} we know that $(-)^\circ$ as a functor
  $(\Rep\Ga^d_k)^\op\to \Rep\Ga^d_k$ is self-adjoint. Passing to
  complexes, one obtains the following pairs of adjoint functors.
  \[\xymatrix{\D(\Rep\Ga^d_k)^\op\ar@<-0.75ex>[r]_-{\mathbf p^\op}&
    \K(\Rep\Ga^d_k)^\op\ar@<-0.75ex>[l]_-{Q^\op}\ar@<-0.75ex>[rr]_-{(-)^\circ}&&
    \K(\Rep\Ga^d_k)\ar@<-0.75ex>[ll]_-{(-)^\circ}\ar@<-0.75ex>[r]_-{Q}&\D(\Rep\Ga^d_k)
    \ar@<-0.75ex>[l]_-{\mathbf p}}
\]
The assertion now follows.
\end{proof}

\subsection*{Hom-tensor identities}
We collect some basic identities for the internal products
$-\lotimes_{\Ga^d_k}-$ and $\RHOM_{\Ga^d_k}(-,-)$.

\begin{lem}\label{le:ltensor-rhom}
Given $X,Y,Z\in\D(\Rep\Ga^d_k)$, the natural morphism
\[\RHOM _{\Ga^d_k} (X,Y) \lotimes_{\Ga^d_k}
Z\lto\RHOM_{\Ga^d_k}(X,Y\lotimes_{\Ga^d_k} Z)\] is an isomorphism
provided that $X$ or $Z$ is perfect.
\end{lem}
\begin{proof}
Adapt the proof of Lemma~\ref{le:tensor-hom}.
\end{proof}

\begin{lem}\label{le:ddual}
  Let $X,Y$ be in $\D(\Rep\Ga^d_k)$ and suppose that $X$ is
  perfect. Then there are natural isomorphisms
\begin{align*}
X\lotimes_{\Ga^d_k}Y^\diamond&\cong \RHOM_{\Ga^d_k}(X,Y)^\diamond\\
(X\lotimes_{\Ga^d_k}Y)^\diamond&\cong \RHOM_{\Ga^d_k}(X,Y^\diamond).
\end{align*}
\end{lem}
\begin{proof}
Adapt the proof of Lemma~\ref{le:dual}.
\end{proof}

\begin{lem}\label{le:ddual1}
Given $X,Y\in\D(\Rep\Ga^d_k)$, there is a natural isomorphism 
\[\RHOM_{\Ga^d_k}(X,Y^\diamond)\cong\RHOM_{\Ga^d_k}(Y,X^\diamond).\]
\end{lem}
\begin{proof}
Adapt the proof of Lemma~\ref{le:dual1}, using Lemma~\ref{le:der-dual}.
\end{proof}

\subsection*{Koszul duality}

We compute $\La^d\lotimes_{\Ga^d_k}-$ and $\RHOM_{\Ga^d_k}(\La^d,-)$ using the
relation between exterior and symmetric powers.

\begin{lem}\label{le:resol}
  Given $d\ge 1$ and $V\in\P_k$, the normalised bar resolutions
  computing $\Ext_{S(V^*)}(k,k)\cong\La V$ and
  $\Ext_{\La(V^*)}(k,k)\cong S V$ yield exact sequences
\begin{multline*}
0\to\Ga^d V\to\bigoplus_{i_1+i_2=d}\Ga^{i_1}V\otimes
\Ga^{i_2}V\to\cdots\\
\to\bigoplus_{i_1+\cdots +i_{d-1}=d}\Ga^{i_1}V\otimes\cdots\otimes
\Ga^{i_{d-1}}V\to V^{\otimes d}\to \La^d V\to 0
\end{multline*}
\begin{multline*}
0\to\La^d V\to\bigoplus_{i_1+i_2=d}\La^{i_1}V\otimes
\La^{i_2}V\to\cdots\\
\to\bigoplus_{i_1+\cdots +i_{d-1}=d}\La^{i_1}V\otimes\cdots\otimes
\La^{i_{d-1}}V\to V^{\otimes d}\to S^d V\to 0
\end{multline*}
which are natural in $V$ and where  $i_1,i_2,\ldots$ run through all positive integers.
\end{lem}
\begin{proof}
  We refer to \cite[p.~359]{Ak1989} or \cite[p.~6]{To1997} for the
  proof and sketch the argument for the convenience of the
  reader. Take the normalised bar resolution \cite[Chapter~X]{ML1975}
  of $k$ over $S(V^*)$:
\[\cdots\to S(V^*)\otimes S^{>0}(V^*) ^{\otimes
  2}\to S(V^*)\otimes S^{>0}(V^*)\to S(V^*)\to k\to 0\] 
Then apply $\Hom_{S(V^*)}(-,k)$:
\[0\to k\lto S^{>0}(V^*)^*\to (S^{>0}(V^*)^*)^{\otimes 2}\to\cdots\]
The cohomology of this complex gives $\Ext_{S(V^*)}(k,k)$ and
identifies with $\La V$ via
classical Koszul duality. Taking the degree $d$ part (with $S^i(V^*)^*$ replaced by
$\Ga^iV$) yields the  resolution of $\La^dV$. The construction of the
resolution of $S^dV$ is analogous.
\end{proof}

The resolution of $\La^d$ shows  that it is a perfect object in $\D(\Rep\Ga^d_k)$ when
viewed as a complex concentrated in degree zero.

\begin{prop}\label{pr:der-exterior}
Let $d\ge 0$. Then
\[\La^d\lotimes_{\Ga^d_k}\La^d\cong S^d.\]
\end{prop}
\begin{proof}
  Tensoring the projective resolution of $\La^d$ in
  Lemma~\ref{le:resol} with $\La^d$ yields an exact sequence which
  identifies with the resolution of $S^d$ via the isomorphism of
  Proposition~\ref{pr:int-ext}.
\end{proof}

\begin{thm}\label{th:duality}
The functors $\La^d\lotimes_{\Ga^d_k}-$ and $\RHOM_{\Ga^d_k}(\La^d,-)$
provide mutually quasi-inverse equivalences
\[\D(\Rep\Ga^d_k)\stackrel{\sim}\lto\D(\Rep\Ga^d_k).\]
\end{thm}
\begin{proof}
  Using Hom-tensor identities, the isomorphism
  $(\La^d)^\diamond\cong\La^d$, and Proposition~\ref{pr:der-exterior},
  one computes 
  \[ \RHOM_{\Ga^d_k}(\La^d,\La^d)\cong (\La^d
  \lotimes_{\Ga^d_k}\La^d)^\diamond \cong (S^d)^\diamond\cong\Ga^d\] and
  this gives for a complex $X$
\begin{align*}
  \RHOM_{\Ga^d_k}(\La^d,\La^d\lotimes_{\Ga^d_k}X) &\cong
  \RHOM_{\Ga^d_k}(\La^d,\La^d)\lotimes_{\Ga^d_k}X\\
   &\cong
  \Ga^d\lotimes_{\Ga^d_k}X\\
  &\cong X \intertext{and}
  \La^d\lotimes_{\Ga^d_k}\RHOM_{\Ga^d_k}(\La^d, X)
  &\cong \RHOM_{\Ga^d_k}(\La^d,\La^d\lotimes_{\Ga^d} X)\\
&\cong  \RHOM_{\Ga^d_k}(\La^d,\La^d) \lotimes_{\Ga^d_k} X\\
&\cong   \Ga^d\lotimes_{\Ga^d_k} X\\
  &\cong X.\tag*{\qed}
\end{align*}
\renewcommand{\qedsymbol}{}
\end{proof}

The following consequence is the contravariant version of Koszul duality.

\begin{cor}\label{co:dualdual}
Given $X\in\D(\Rep\Ga^d_k)$, there is a natural isomorphism
\[\RHOM_{\Ga^d_k}(\RHOM_{\Ga^d_k}(X,\La^d),\La^d)\cong
X^{\diamond\diamond}.\]
\end{cor}
\begin{proof} 
  We combine Hom-tensor identities, the isomorphism
  $(\La^d)^\diamond\cong\La^d$, and Theorem~\ref{th:duality}. This
  gives
\begin{align*}
  \RHOM_{\Ga^d_k}(\RHOM_{\Ga^d_k}(X,\La^d),\La^d)&\cong
  \RHOM_{\Ga^d_k}(\RHOM_{\Ga^d_k}(\La^d,X^\diamond),\La^d)\\
 &\cong
  \RHOM_{\Ga^d_k}(\La^d,\RHOM_{\Ga^d_k}(\La^d,X^\diamond)^\diamond)\\
&\cong
  \RHOM_{\Ga^d_k}(\La^d,\La^d\lotimes_{\Ga^d_k}X^{\diamond\diamond})\\
&\cong X^{\diamond\diamond}.  \tag*{\qed}
\end{align*}
\renewcommand{\qedsymbol}{}
\end{proof}

\subsection*{The bounded derived category of finite representations}

Let $\D^b(\rep\Ga^d_k)$ denote the bounded derived category of the
exact category $\rep\Ga^d_k$.

\begin{lem}
The inclusion $\rep\Ga^d_k\to\Rep\Ga^d_k$ induces a fully faithful and
exact functor
\[\D^b(\rep\Ga^d_k)\lto\D(\Rep\Ga^d_k).\]
\end{lem}
\begin{proof}
  The proof of Theorem~\ref{th:schur} shows that $\Ga^{d,k^d}$ is a
  projective generator of $\Rep\Ga^d_k$ which belongs to
  $\rep\Ga^d_k$. Thus each object $X$ in $\rep\Ga^d_k$ admits an
  epimorphism $P\to X$ such that $P$ is a projective object and
  belongs to $\rep\Ga^d_k$. It follows that the inclusion
  $\rep\Ga^d_k\to\Rep\Ga^d_k$ induces a fully faithful functor
  $\D^b(\rep\Ga^d_k)\to\D(\Rep\Ga^d_k)$; see for instance
  \cite[Proposition~III.2.4.1]{Ver}.
\end{proof}

The duality $(\rep\Ga^d_k)^\op\xto{\sim}\rep\Ga^d_k$ taking $X$ to
$X^\circ$ is an exact functor and induces therefore a duality
\[\D^b(\rep\Ga^d_k)^\op\stackrel{\sim}\lto\D^b(\rep\Ga^d_k).\] Note that
$X^\diamond\cong X^\circ$ for all $X\in\D^b(\rep\Ga^d_k)$. 

Given $X,Y\in\D^b(\rep\Ga^d_k)$, the objects $X\lotimes_{\Ga^d_k}Y$
and $\RHOM_{\Ga^d_k}(X,Y)$ belong to $\D^b(\rep\Ga^d_k)$ provided that
$X$ is perfect. This observation yields the following immediate
consequence of Theorem~\ref{th:duality}.

\begin{cor}\label{co:duality}
  \pushQED{\qed} The functors $\La^d\lotimes_{\Ga^d_k}-$ and
  $\RHOM_{\Ga^d_k}(\La^d,-)$ induce mutually quasi-inverse
  equivalences \[\D^b(\rep\Ga^d_k)\stackrel{\sim}\lto\D^b(\rep\Ga^d_k).\qedhere\]
\end{cor}

A consequence of Corollary~\ref{co:dualdual} is the following.

\begin{cor}
  The functor $\RHOM_{\Ga^d_k}(-,\La^d)$ induces an equivalence
  \[D\colon\D^b(\rep\Ga^d_k)^\op\stackrel{\sim}\lto\D^b(\rep\Ga^d_k)\] satisfying
  $D^2\cong\Id$.\qed
\end{cor}

\subsection*{Finite global dimension}

The Schur algebra $S_k(n,d)$ has finite global dimension provided that
$k$ is a field or $k=\bbZ$. This is a classical fact \cite{AB1988,
  Do1986} and follows also from the more general fact that Schur
algebras are quasi-hereditary; precise dimensions were computed by
Totaro \cite{To1997}. We extend this result as follows.

\begin{prop}\label{pr:gldim}
  Suppose that $k$ is noetherian and let $X,Y\in\rep\Ga^d_k$. Then
  $\Ext^i_{\Ga^d_k}(X,Y)=0$ for all $i > 2d$.
\end{prop}

The proof is based on the following elementary lemma which is taken
from \cite{CPS1990}; it also explains the noetherianess
hypothesis. For a prime ideal $\frp\subseteq k$, let $k(\frp)$ denote
the residue field $k_\frp/{\frp_\frp}$.

\begin{lem}\label{le:gldim}
  Let $A$ be a noetherian $k$-algebra and $M$ a finitely generated
  $A$-module. Suppose that $A$ and $M$ are $k$-projective.  Then $M$
  is projective over $A$ if and only if $M\otimes _k k(\frp)$ is
  projective over $A\otimes _k k(\frp)$ for all prime ideals
  $\frp\subseteq k$.
\end{lem}
\begin{proof}
  One direction is clear. So suppose that $M\otimes _k k(\frp)$
  is projective over $A\otimes _k k(\frp)$ for all $\frp$.
  It suffices to prove the assertion when $k$ is local with maximal
  ideal $\frm$, and we may assume that $k$ is complete since $k$ is
  noetherian. Thus $A$ is semi-perfect and a projective cover $P\to
  M\otimes _k k(\frm)$ lifts to a projective cover $P\to M$, which is
  an isomorphism since $P\otimes _k k(\frm)\to M\otimes _k k(\frm)$ is
  one. It follows that $M$ is projective over $A$.
\end{proof}

\begin{proof}[Proof of Proposition~\ref{pr:gldim}]
We work in the category of modules over $S_k(d,d)$ which is equivalent to
$\Rep\Ga^d_k$ by Theorem~\ref{th:schur}. Note that for all $n\ge 1$
and each ring homomorphism $k\to k'$
\[S_{k'}(n,d)\cong S_k(n,d)\otimes_k k'.\] This is a consequence of the
base change formula \[\Ga_{k'}^d(V\otimes_k k')\cong (\Ga^d_k
V)\otimes_k k'\] 
which holds for each $V\in\P_k$. The global dimension of $S_{k(\frp)}(d,d)$ is
bounded by $2d$ for all prime ideals $\frp\subseteq k$, by
\cite[Theorem~2]{To1997}. Thus the assertion follows from
Lemma~\ref{le:gldim}.
\end{proof}

\subsection*{Schur and Weyl functors}

Fix a positive integer $d$. A \emph{partition} of \emph{weight} $d$ is
a sequence $\la=(\la_1,\la_2,\ldots )$ of non-negative integers
satisfying $\la_1\ge\la_2\ge \ldots$ and $\sum\la_i=d$. Its
\emph{conjugate} $\la'$ is the partition where $\la'_i$ equals the
number of terms of $\la$ that are greater or equal than $i$.  

Fix a partition $\la$ of weight $d$. Each integer $r\in\{1,\ldots,d\}$
can be written uniquely as sum $r=\la_1+\ldots\la_{i-1}+j$ with $1\le j\le
\la_i$. The pair $(i,j)$ describes the position ($i$th row and $j$th
column) of $r$ in the \emph{Young diagram} corresponding to $\la$. The
partition $\la$ determines a permutation $\s_\la\in\frS_d$ by
$\s_\la(r)=\la'_1+\ldots\la'_{j-1}+i$, where $1\le i\le\la_j$. Note
that $\s_{\la'}=\s_\la^{-1}$.

Fix a partition $\la$ of weight $d$, and assume that $\la_1+\cdots
+\la_n=d=\la'_1+\cdots +\la'_m$. Following \cite[II.1]{ABW1982}, there
is defined for $V\in\P_k$ the \emph{Schur module} $S_\la V$ as image
of the map
\[\La^{\la'_1}V\otimes\cdots\otimes
\La^{\la'_m}V\xto{\Delta\otimes\cdots\otimes\Delta} V^{\otimes
  d}\xto{s_{\la'}}V^{\otimes d}\xto{\nabla\otimes\cdots\otimes\nabla}
S^{\la_1}V\otimes\cdots\otimes S^{\la_n}V.\] Here, we denote for an
integer $r$ by $\Delta\colon\La^{r}V\to V^{\otimes r}$ the
component of the comultiplication given by
\[\Delta(v_1\wedge\cdots\wedge v_{r})=\sum_{\s\in\frS_{r}}
{\sgn(\s)}v_{\s(1)}\otimes\cdots\otimes v_{\s(r)},\]
$\nabla\colon V^{\otimes r}\to S^r V$ is the multiplication, and
$s_\la\colon V^{\otimes d}\to V^{\otimes d}$ is given by
\[s_\la(v_1\otimes\cdots\otimes v_d)=v_{\s_\la(1)}\otimes\cdots\otimes
v_{\s_\la(d)}.\]
The corresponding \emph{Weyl module} $W_\la$ is by definition the
image of the analogous map
\[\Ga^{\la_1}V\otimes\cdots\otimes
\Ga^{\la_n}V\xto{\Delta\otimes\cdots\otimes\Delta} V^{\otimes
  d}\xto{s_\la}V^{\otimes d}\xto{\nabla\otimes\cdots\otimes\nabla}
\La^{\la'_1}V\otimes\cdots\otimes \La^{\la'_m}V.\] Note that
$(W_\la V)^*\cong S_{\la}(V^*)$. Moreover, $S_\la V$ is free when $V$
is free, by \cite[Theorem~II.2.16]{ABW1982}. Thus $S_\la V$ and $W_\la
V$ belong to $\P_k$ for all $V\in\P_k$.

The definition of Schur and Weyl modules gives rise to functors $S_\la$
and $W_\la$ in $\Rep\Ga^d_k$ for each partition $\la$ of weight $d$.
In \cite[\S4]{AB1988}, a finite resolution $\mathbf\Gamma(W_\la)$ of
$W_\la$ in terms of divided powers is constructed, and in
\cite[Theorem~6.1]{AB1988} it is shown that
$\La^d\otimes_{\Ga^d_k}\mathbf\Gamma(W_\la)$ is a resolution of
$S_{\la'}$, basically using an explicit description of the functor
\[\add\{\Ga^\mu\mid\mu\in\La(n,d)\}\stackrel{\sim}\lto
\add\{\La^\mu\mid\mu\in\La(n,d)\}\] from Corollary~\ref{co:equiv}.
Summarising this discussion, we have the following result.

\begin{prop}\label{pr:schur}
\pushQED{\qed} For each partition $\la$ of weight $d$, there is an isomorphism
\[\La^d\lotimes_{\Ga^d_k}W_\la\cong S_{\la'}.\qedhere\]
\end{prop}

The functor $\La^d\lotimes_{\Ga^d_k}-$ gives an equivalence
$\D(\Rep\Ga^d_k)\xto{\sim}\D(\Rep\Ga^d_k)$, by
Theorem~\ref{th:duality}. Thus the classical formula
\[\Ext_{\Ga^d_k}^*(W_\la,W_\mu)\cong\Ext_{\Ga^d_k}^*(S_{\la'},S_{\mu'})\]
due to Akin--Buchsbaum \cite[Theorem~7.7]{AB1988} and Donkin
\cite[Corollary~3.9]{Do1993} follows.

\subsection*{An explicit example}

Set $k=\bbF_2$. We compute $\rep\Ga^2_k$. This is an abelian
\emph{length category}, that is, each object has finite composition
length. The simple objects are indexed by the partitions $\a=(2)$ and
$\om=(1,1)$. We describe the indecomposable objects by specifying the
factors of a composition series; see also \cite{Er1993} for a
description of the corresponding Schur algebra which is of finite
representation type.  The indecomposable projective objects are
\[\Ga^2=\Ga^\a=\sao\quad\text{and}\quad \Ga^\om=\soao\] 
and the indecomposable injective objects are
\[S^2=S^\a=\soa\quad\text{and}\quad S^\om=\soao.\]
The resolutions from Lemma~\ref{le:resol} have the form
\begin{gather*}
0\lto\Ga^\a\lto\Ga^\om\lto\La^\a\lto 0\\
0\lto\La^\a\lto\La^\om\lto S^\a\lto 0
\end{gather*}
and therefore
\[\La^2=\La^\a=\so\quad\text{and}\quad \La^\om=\soao.\]
The Schur and Weyl functors are 
\[S_\a=\soa\qquad S_\om=\so\qquad W_\a=\sao\qquad W_\om=\so.\]
The decomposition \eqref{eq:decomp} of $\Ga^{2,k^2}_k$ has the form
\[\Ga^{2,k^2}_k=\Ga^\a\oplus\Ga^\om\oplus\Ga^\a\]
and the endomorphism algebra of $\Ga^{2,k^2}_k$ is the Schur algebra
$S_k(2,2)$. The functor $\Hom_{\Ga^2_k}(\Ga^{2,k^2}_k,-)$ induces an
equivalence $\rep\Ga^2_k\xto{\sim} \mod S_k(2,2)$.  The following
table gives for each pair of indecomposable representations in
$\rep\Ga^2_k$ their tensor product.
\begin{figure}[!h]
\[\begin{array}{|c||ccccc|}
\hline
\otimes_{\Ga^2_k}&\sa&\so&\sao&\soa&\soao\\
\hline\hline
\sa&\sa&0&\sa&0&0\\
\so&0&\soa&\so&\soa&\soao\\
\sao&\sa&\so&\sao&\soa&\soao\\
\soa&0&\soa&\soa&\soa&\soao\\
\soao&0&\soao&\soao&\soao&\soao\!\oplus\!\soao\\
\hline
\end{array}\]
\end{figure}

I am grateful to Dieter Vossieck for pointing out the following
connection.

\begin{rem}[(Vossieck)]
  Let $k[\e]$ denote the $k$-algebra of dual numbers ($\e^2=0$). The
  category $\rep\Ga^2_k$ is equivalent to the category of finitely
  generated modules over the Auslander algebra \cite{Au1971} of
  $k[\e]$. The category of $k[\e]$-modules admits (at least) three
  different tensor products which induce tensor products on the
  category of modules over the Auslander algebra (via Day
  convolution): We have $-\otimes_{k[\e]}-$ and $-\otimes_k-$, where
  in the second case the action of $k[\e]$ is either given by
  $\e\mapsto \e\otimes 1 + 1\otimes\e+\e\otimes\e$ (the group scheme
  $\mu_2$) or by $\e\mapsto \e\otimes 1 + 1\otimes\e$ (the group
  scheme $\a_2$). The last case yields a tensor product which is
  equivalent to $-\otimes_{\Ga^2_k}-$.
\end{rem}

\section{Koszul versus Ringel and Serre duality}

In this section we explain the connection between the Koszul duality
from Theorem~\ref{th:duality} and Ringel duality for Schur algebras,
as developed in work of Ringel \cite{Ri1991} and Donkin
\cite{Do1993}. Then we show that Koszul duality applied twice yields a
Serre functor in the sense of Reiten and Van den Bergh \cite{RV2002}.

\subsection*{Ringel duality}

Fix integers $d,n\ge 1$ with $n\ge d$ and recall from
Theorem~\ref{th:schur} that evaluation at $k^n$ gives an
equivalence
\[\Hom_{\Ga^d_k}(\Ga^{d,k^n},-)\colon\Rep\Ga^d_k\stackrel{\sim}\lto\Mod S_k(n,d).\]
The module
\[T=\Hom_{\Ga^d_k}(\Ga^{d,k^n},\La^{d,k^n})\] is the characteristic
tilting module\footnote{A module $T$ over any ring $A$ is a
  \emph{tilting module} iff $\RHom_A(T,-)$ induces an equivalence
  $\D(\Mod A)\xto{\sim} \D(\Mod\End_A(T))$. If $A$ is
  quasi-hereditary, then this extra structure determines a
  \emph{characteristic tilting module}. Its endomorphism ring $A'$ is
  called \emph{Ringel dual} of $A$; it is again quasi-hereditary and
  $A''$ is Morita equivalent to $A$; see \cite{Ri1991}. Note that
  Schur algebras are quasi-hereditary.}  for the Schur algebra
$S_k(n,d)$; see \cite{Do1993,Ri1991}. In particular, it induces an
equivalence
\[\RHom_{S_k(n,d)}(T,-)\colon\D(\Mod S_k(n,d))\stackrel{\sim}\lto\D(\Mod\End_{S_k(n,d)}(T)).\] 
The connection between this equivalence and the equivalence of
Theorem~\ref{th:duality} is explained as follows. First observe that
\[\La^{d,k^n}=\La^d\otimes_{\Ga^d_k}\Ga^{d,k^n}.\] 
It follows that the composite of the equivalences
$\La^d\lotimes_{\Ga^d_k}-$ and $\Hom_{\Ga^d_k}(\Ga^{d,k^n},-)$ induces an
isomorphism
\[\p\colon
S_k(n,d)=\End_{\Ga^d_k}(\Ga^{d,k^n})\xto{\sim}\End_{\Ga^d_k}(\La^{d,k^n}) \xto{\sim}\End_{S_k(n,d)}(T).\]

\begin{thm}\label{th:Ringel}
The following diagram commutes up to a natural isomorphism.
\[\xymatrixrowsep{3pc} \xymatrixcolsep{2.5pc} 
\xymatrix{\D(\Rep\Ga^d_k)\ar[d]_-\wr^-{\RHom(\Ga^{d,k^n},-)}\ar[rrr]_-\sim^-{\RHOM(\La^d,-)}&&&\D(\Rep\Ga^d_k)\ar[d]^-\wr_-{\RHom(\Ga^{d,k^n},-)}\\
  \D(\Mod S_k(n,d)) \ar[rr]_-\sim^-{\RHom(T,-)}&&\D(\Mod\End(T))
  \ar[r]_-\sim^-{\p_*}&\D(\Mod S_k(n,d)) }\]
\end{thm}
\begin{proof}
We have
\begin{align*}
\RHom_{\Ga^d_k}(\Ga^{d,k^n},\RHOM_{\Ga^d_k}(\La^d,-))&\cong\RHom_{\Ga^d_k}(\La^{d,k^n},-)\\
&\cong \p_*\RHom_{S_k(n,d)}(T,\RHom_{\Ga^d_k}(\Ga^{d,k^n},-))
\end{align*}
where the first isomorphism is adjunction and the second is induced by
evaluation at $k^n$.
\end{proof}

\begin{rem}
It follows from Theorems~\ref{th:duality} and \ref{th:Ringel} that $T$
is a tilting module over the Schur algebra $S_k(n,d)$.
\end{rem}

\subsection*{Serre duality}

Fix an integer $d\ge 0$ and suppose that $k$ is a field. Our aim is to
describe a Serre duality functor for $\D^b(\rep\Ga^d_k)$.  

Recall from work of Reiten and Van den Bergh \cite{RV2002} the
following notion of Serre duality.  Let $\C$ be a $k$-linear
triangulated category and suppose that $\Hom_\C(X,Y)$ is finite
dimensional over $k$ for all objects $X,Y$ in $\C$.  A \emph{Serre
  functor} is by definition an equivalence $F\colon\C\xto{\sim}\C$
together with a natural isomorphism
\[\Hom_\C(X,-)^*\xto{\sim}\Hom_\C(-,FX)\]
for each $X\in\C$. When a Serre functor exists, it is unique up to isomorphism.

Our description of a Serre functor is based on the following lemma.

\begin{lem}\label{le:serre}
Let $X,Y$ be in $\Rep\Ga^d_k$ and suppose that $X$ is finitely
generated projective. Then there is a natural isomorphism
\[\Hom_{\Ga^d_k}(X,Y)^*\cong\Hom_{\Ga^d_k}(Y,S^d\otimes_{\Ga^d_k}
X).\]
\end{lem}
\begin{proof}
Using Hom-tensor identities, we compute for $X=\Ga^{d,V}$
\begin{align*}
\Hom_{\Ga^d_k}(\Ga^{d,V},Y)^*&\cong Y(V)^*\\
&\cong \Hom_{\Ga^d_k}(\Ga^{d,V^*}, Y^\circ)\\
&\cong\Hom_{\Ga^d_k}((S^d\otimes_{\Ga^d_k}\Ga^{d,V})^\circ, Y^\circ)\\
&\cong\Hom_{\Ga^d_k}(Y,S^d\otimes_{\Ga^d_k}\Ga^{d,V}).\tag*{\qed}
\end{align*}
\renewcommand{\qedsymbol}{}
\end{proof}

\begin{prop}\label{pr:serre}
  Let $k$ be a field.  Given $X,Y\in\D^b(\rep\Ga^d_k)$, there is a
  natural isomorphism
\[\Hom_{\D(\Ga^d_k)}(X,Y)^*\cong\Hom_{\D(\Ga^d_k)}(Y,S^d\lotimes_{\Ga^d_k}
X).\]
\end{prop}
\begin{proof}
  One can assume that $X$ and $Y$ are bounded complexes of finitely
  generated projective objects, by Proposition~\ref{pr:gldim}.  Form
  the total complex $\Hom_{\Ga^d_k}(X,Y)$ defined in degree $n$ by
\[\Hom_{\Ga^d_k}(X,Y)^n=\prod_{p+q=n} \Hom_{\Ga^d_k}(X^{-p},Y^q).\]
Note that this involves in each degree a product having only finitely
many non-zero factors.  Taking cohomology in degree zero
gives \[H^0\Hom_{\Ga^d_k}(X,Y)\cong \Hom_{\K(\Ga^d_k)}(X,Y)\cong
\Hom_{\D(\Ga^d_k)}(X,Y),\] where the first isomorphism is a general
fact while the second uses that $X$ is perfect. Now the assertion
follows from the isomorphism in Lemma~\ref{le:serre}, using that the
duality with respect to $k$ is exact.
\end{proof}

In Corollary~\ref{co:duality} it has been shown that
$\La^d\lotimes_{\Ga^d_k}-$ induces an equivalence
\[\D^b(\rep\Ga^d_k)\stackrel{\sim}\lto\D^b(\rep\Ga^d_k)\]
and it would be interesting to have a description of its powers
$(\La^d\lotimes_{\Ga^d_k}-)^n$. The first step is the following case $n=2$.

\begin{cor}\label{co:serre}
Let $k$ be a field. The functor $(\La^d\lotimes_{\Ga^d_k}-)^2\cong S^d\lotimes_{\Ga^d_k}-$
induces a Serre functor $\D^b(\rep\Ga^d_k)\xto{\sim}\D^b(\rep\Ga^d_k)$.
\end{cor}
\begin{proof}
Combine Propositions~\ref{pr:der-exterior} and \ref{pr:serre}.
\end{proof}

I am grateful to Bernhard Keller and Claus Ringel for pointing out
some different perspectives.

\begin{rem}[(Keller)]
  A perhaps related phenomenon was observed in Section~10.3 of
  \cite{Ke1994}: If $A$ is an augmented dg algebra and $A^*$ its
  Koszul dual, there is a canonical triangle functor $F_A$ from the
  derived category of $A^*$ to that of $A$ and a canonical triangle
  functor $\can_A$ from the derived category of $A$ to that of
  $A^{**}$.  With these notations, there is a canonical morphism $F_A
  F_{A^*} \can_A \to S$, where $S$ is the `Serre functor' of $A$,
  i.e.\ the derived tensor product by the $k$-dual bimodule of $A$.
\end{rem}

\begin{rem}[(Ringel)]
  Let $A$ be a quasi-hereditary algebra, with indecomposable projective
  modules $P(i)$, and indecomposable injective modules $Q(i)$. Let $T$
  be the characteristic tilting module with indecomposable summands
  $T(i)$. Let $B =\End_A(T)$. Then $\Hom_A(T,-)$ sends the
  $\nabla$-filtered $A$-modules to the $\Delta$-filtered $B$-modules,
  in particular it sends $Q(i)$ to $\Hom_A(T,Q(i))$, which is a direct
  summand of the characteristic tilting module for $B$.  Now assume
  that $B = A$ (as quasi-hereditary algebras). Thus $F = \RHom_A(T,-)$
  is an auto-equivalence of $\D^b(\mod A)$.  As we just have seen, $F$
  sends $Q(i)$ to $T(i)$.  But of course $F$ sends $T(i)$ to $P(i)$.
  Thus $F^2$ sends $Q(i)$ to $P(i)$ which is just the (inverse of the)
  Serre functor.
\end{rem}

\end{document}